\documentclass[10pt]{amsart}
\usepackage{amssymb}%,MnSymbol,times}
\usepackage{amsthm,amsmath}

\title{On indefinite sums weighted by periodic sequences}

\author{Jean-Luc Marichal}\thanks{Jean-Luc Marichal is with the Mathematics Research Unit, University of Luxembourg, Maison du Nombre, 6, avenue de la Fonte, L-4364 Esch-sur-Alzette, Luxembourg.\\ Email: jean-luc.marichal[at]uni.lu}
\address{Mathematics Research Unit, University of Luxembourg, Maison du Nombre, 6, avenue de la Fonte, L-4364 Esch-sur-Alzette, Luxembourg}
\email{jean-luc.marichal[at]uni.lu}

\date{October 15, 2018}

\theoremstyle{plain}
\newtheorem{theorem}{Theorem}%[section]

\newtheorem{proposition}[theorem]{Proposition}

\newtheorem{fact}[theorem]{Fact}

\theoremstyle{definition}

\newtheorem{example}[theorem]{Example}

\theoremstyle{remark}

\newtheorem{remark}{Remark}

\newcommand{\R}{\mathbb{R}}
\newcommand{\C}{\mathbb{C}}
\newcommand{\N}{\mathbb{N}}
\newcommand{\Z}{\mathbb{Z}}

\begin{document}
\begin{abstract}
For any integer $q\geq 2$ we provide a formula to express indefinite sums of a sequence $(f(n))_{n\geq 0}$ weighted by $q$-periodic sequences in terms of indefinite sums of sequences $(f(qn+p))_{n\geq 0}$, where $p\in\{0,\ldots,q-1\}$. When explicit expressions for the latter sums are available, this formula immediately provides explicit expressions for the former sums. We also illustrate this formula through some examples.
\end{abstract}



\maketitle

%---------------------------------------------------------------------------------------
\section{Introduction}

Let $\N=\{0,1,2,\ldots\}$ be the set of non-negative integers. We assume throughout that $q\geq 2$ is a fixed integer and we set $\omega=\exp(\frac{2\pi i}{q})$. Consider two functions $f\colon\N\to\C$ and $g\colon\Z\to\C$ and suppose that $g$ is $q$-periodic, that is, $g(n+q)=g(n)$ for every $n\in\Z$. Also, consider the functions $S\colon\N\to\C$ and $T_p\colon\N\to\C$ ($p=0,\ldots,q-1$) defined by
$$
S(n) ~=~ \sum_{k=0}^{n-1}g(k) f(k)
$$
and
$$
T_p(n) ~=~ \sum_{k=0}^{n-1}f(q k+p),
$$
respectively. These functions are indefinite sums (or anti-differences) in the sense that the identities
$$
\Delta_n S(n) ~=~ g(n)f(n)\qquad\text{and}\qquad \Delta_n T_p(n) ~=~ f(q n+p)
$$
hold on $\N$, where $\Delta_n$ is the classical difference operation defined by $\Delta_n f(n)=f(n+1)-f(n)$; see, e.g., \cite[{\S}2.6]{GraKnuPat94}.

In this paper we provide a conversion formula that expresses the sum $S(n)$ in terms of the sums $T_p(n)$ for $p=0,\ldots,q-1$ (see Proposition~\ref{prop:main}). Such a formula can sometimes be very helpful for it enables us to find an explicit expression for $S(n)$ whenever explicit expressions for the sums $T_p(n)$ for $p=0,\ldots,q-1$ are available.

\begin{example}\label{ex:1}
The sum
$$
S(n) ~=~ \sum_{k=1}^{n-1}\cos\left(\frac{2k\pi}{3}\right)\log k{\,},\qquad n\in\N\setminus\{0\},
$$
where the function $g(k)=\cos(\frac{2k\pi}{3})$ is $3$-periodic, can be computed from the sums
$$
T_p(n) ~=~ \sum_{k=1}^{n-1}\log(3k+p) ~=~ \log\left(3^{n-1}\,\frac{\Gamma(n+p/3)}{\Gamma(1+p/3)}\right),\qquad p\in\{0,1,2\}.
$$
Using our conversion formula we arrive, after some algebra, at the following closed-form representation
$$
S(n) ~=~ \frac{1}{4}\,\log\left(\frac{4\pi^2}{27}\right)+\sum_{j=0}^2\cos\left(\frac{2\pi (n+j)}{3}\right)\,\log\left(3^{j/3}\,\Gamma\left(\frac{n+j}{3}\right)\right).
$$
\end{example}

Oppositely, we also provide a very simple conversion formula that expresses each of the sums $T_p(n)$ for $p=0,\ldots,q-1$ in terms of sums of type $S(n)$ (see Proposition~\ref{prop:main}). As we will see through a couple of examples, this formula can sometimes enable us to find closed-form representations of sums $T_p(n)$ that are seemingly hard to evaluate explicitly.

\begin{example}
Sums of the form
$$
\sum_{k\geq 0}{m\choose q k+p}\frac{1}{q k+p+1}
$$
where $m\in\N$ and $p\in\{0,\ldots,q-1\}$, may seem complex to be evaluated explicitly. In Example~\ref{ex:bin} we provide a closed-form representation of this sum by first considering a sum of the form
$$
\sum_{k\geq 0}g(k){m\choose k}\frac{1}{k+1},
$$
for some $q$-periodic function $g(k)$. For instance, we obtain
$$
\sum_{k\geq 0}{m\choose 3 k+1}\frac{1}{3k+2} ~=~ \frac{1}{6(m+1)}{\,}\left(2^{m+2}-3\cos\frac{m\pi}{3}-\cos\frac{5m\pi}{3}\right).
$$
\end{example}

In this paper we also provide an explicit expression for the (ordinary) generating function of the sequence $(S(n))_{n\geq 0}$ in terms of the generating function of the sequence $(f(n))_{n\geq 0}$ (see Proposition~\ref{prop:gf}). Finally, in Section 3 we illustrate our results through some examples.

%---------------------------------------------------------------------------------------
\section{The result}

The functions $g_p\colon\Z\to\{0,1\}$ ($p=0,\ldots,q-1$) defined by
$$
g_p(n) ~=~ g_0(n-p) ~=~
\begin{cases}
1, & \text{if $n=p ~(\mathrm{mod}~q)$},\\
0, & \text{otherwise},
\end{cases}
$$
form a basis of the linear space of $q$-periodic functions $g\colon\Z\to\C$. More precisely, for any $q$-periodic function $g\colon\Z\to\C$, we have
\begin{equation}\label{eq:gngpn}
g(n) ~=~ \sum_{p=0}^{q-1}g(p){\,}g_p(n) ~=~ \sum_{p=0}^{q-1}g(p){\,}g_0(n-p){\,},\qquad n\in\Z.
\end{equation}
To compute the sum $S(n)=\sum_{k=0}^{n-1}g(k) f(k)$, it is then enough to compute the $q$ sums
$$%\begin{equation}\label{eq:Sp}
S_p(n) ~=~ \sum_{k=0}^{n-1}g_p(k)f(k) ~=~ \sum_{k=0}^{n-1}g_0(k-p)f(k){\,},\qquad p=0,\ldots,q-1.
$$%\end{equation}
Indeed, using \eqref{eq:gngpn} we then have
\begin{equation}\label{eq:SSp}
S(n) ~=~ \sum_{p=0}^{q-1}g(p){\,}S_p(n){\,},\qquad n\in\N.
\end{equation}

It is now our aim to find explicit conversion formulas between $S_p(n)$ and $T_p(n)$, for every $p\in\{0,\ldots,q-1\}$. The result is given in Proposition~\ref{prop:main} below. We first consider the following fact that can be immediately derived from the definitions of the sums $S_p$ and $T_p$.

\begin{fact}\label{fact:1}
For any $n\in\N$ and any $p\in\{0,\ldots,q-1\}$, we have
$$
S_p(n) ~=~ T_p(\lfloor(n-p-1)/q\rfloor +1).
$$
In particular, we have $S_p(q n+p-i)= T_p(n)$ for any $i\in\{0,\ldots,q-1\}$ such that $qn+p-i\in\N$.
\end{fact}

For any $p\in\{0,\ldots,q-1\}$, let $T_p^+\colon D_p\to\C$ be any extension of $T_p$ to the set $D_p=\{\frac{-p}{q},\frac{1-p}{q},\frac{2-p}{q},\ldots\}$. By definition, we have $\N\subset D_p$ and $T_p^+=T_p$ on $\N$.

\begin{proposition}\label{prop:main}
For any $n\in\N$ and any $p\in\{0,\ldots,q-1\}$, we have $T_p(n) = S_p(qn)$ and
\begin{equation}\label{eq:thmmain}
S_p(n) ~=~ \sum_{k=0}^{q-1}g_0(n+k-p){\,}T_p^+\left(\frac{n+k-p}{q}\right).
\end{equation}
\end{proposition}

\begin{remark}
We observe that each summand for which $n+k-p <0$ in \eqref{eq:thmmain} is zero since in this case we have $n+k\neq p~(\mathrm{mod}~q)$.
\end{remark}

\begin{proof}[Proof of Proposition~\ref{prop:main}]
Let $n\in\N$ and $p\in\{0,\ldots,q-1\}$. The identity $T_p(n) = S_p(qn)$ immediately follows from Fact~\ref{fact:1}. Now, for any $k\in\N$ we have $g_0(n+k-p)=1$ if and only if there exists $M\in\Z$ such that $k=Mq+p-n$. Assuming that $k\in\{0,\ldots,q-1\}$, the latter condition holds if and only if $M=\lfloor (n-p-1)/q\rfloor +1$. Thus, the sum in \eqref{eq:thmmain} reduces to $T_p(\lfloor(n-p-1)/q\rfloor +1)$, which is $S_p(n)$ by Fact~\ref{fact:1}.
\end{proof}

\begin{proof}[Alternative proof of \eqref{eq:thmmain}]
The identity clearly holds for $n=0$ since we have $S_p(0)=0=T_p(0)$. It is then enough to show that \eqref{eq:thmmain} still holds after applying the difference operator $\Delta_n$ to each side. Applying $\Delta_n$ to the right-hand side, we immediately obtain a telescoping sum that reduces to
$$
g_0(n+q-p){\,}T_p^+\left(\frac{n+q-p}{q}\right) - g_0(n-p){\,}T_p^+\left(\frac{n-p}{q}\right),
$$
that is,
$$
g_0(n-p){\,}(\Delta{\,}T_p^+)\left(\frac{n-p}{q}\right).
$$
If $n\neq p ~(\mathrm{mod}~q)$, then $g_0(n-p)=0$ and hence the latter expression reduces to zero. Otherwise, it becomes $g_0(n-p)(\Delta{\,}T_p)(\frac{n-p}{q})$. In both cases, the expression reduces to $g_0(n-p){\,}f(n)$, which is nothing other than $\Delta_n S_p(n)$.
\end{proof}

Let $F(z)=\sum_{n\geq 0}f(n){\,}z^n$ and $F_p(z)=\sum_{n\geq 0}f(qn+p){\,}z^n$ be the generating functions of the sequences $(f(n))_{n\geq 0}$ and $(f(qn+p))_{n\geq 0}$, respectively. The following proposition provides explicit forms of the generating function of the sequence $(S_p(n))_{n\geq 0}$ in terms of $F(z)$ and $F_p(z)$. The proof of this proposition uses a familiar trick for extracting alternate terms of a series; see, e.g., \cite[p.~90]{CulFalRon05} and \cite[p.~89]{Knu97}.

Recall that if $A(z)=\sum_{n\geq 0}f(n){\,}z^n$ is the generating function of a sequence $(a(n))_{n\geq 0}$, then $\frac{1}{1-z}A(z)$ is the generating function of the sequence of the partial sums $(\sum_{k=0}^na(k))_{n\geq 0}$; see, e.g., \cite[{\S}5.4]{GraKnuPat94} and \cite[p.~89]{Knu97}.

\begin{proposition}\label{prop:gf}
If $F(z)$ converges in some disk $|z|<R$, then
$$
\sum_{n\geq 0}S_p(n){\,}z^n ~=~ \frac{z}{q(1-z)}{\,}\sum_{k=0}^{q-1}\omega^{-kp}{\,}F(\omega^kz) ~=~ \frac{z^{p+1}}{1-z}{\,}F_p(z^q).
$$
\end{proposition}

\begin{proof}
We first observe that the identity $\frac{1}{q}{\,}\sum_{k=0}^{q-1}\omega^{kn}=g_0(n)$ holds for any $n\in\Z$. For any $z\in\C$ such that $|z|<R$, we then have
\begin{eqnarray*}
\frac{z}{q(1-z)}{\,}\sum_{k=0}^{q-1}\omega^{-kp}{\,}F(\omega^kz)% &=& \frac{z}{1-z}{\,}\sum_{n\geq 0}f(n){\,}z^n\left(\frac{1}{q}{\,}\sum_{k=0}^{q-1}\omega^{k(n-p)}\right)\\
&=& \frac{z}{1-z}{\,}\sum_{n\geq 0}g_p(n)f(n){\,}z^n\\
&=& \sum_{n\geq 0}S_p(n+1){\,}z^{n+1} ~=~ \sum_{n\geq 0}S_p(n){\,}z^n,
\end{eqnarray*}
which proves the first formula. For the second formula, we simply observe that
$$
\sum_{n\geq 0}g_p(n)f(n){\,}z^n ~=~ \sum_{n\geq 0}f(qn+p){\,}z^{qn+p} ~=~ z^p{\,}F_p(z^q).
$$
This completes the proof.
\end{proof}

We end this section by providing explicit forms of the function $g_0(n)$. For instance, it is easy to verify that
$$
g_0(n) ~=~ \left\lfloor\frac{n}{q}\right\rfloor - \left\lfloor\frac{n-1}{q}\right\rfloor ~=~ \Delta_n\left\lfloor\frac{n-1}{q}\right\rfloor.
$$
As already observed in the proof of Proposition~\ref{prop:gf}, we also have
\begin{equation}\label{eq:g01}
g_0(n) ~=~ \frac{1}{q}{\,}\sum_{j=0}^{q-1}\omega^{j n} ~=~ \frac{1}{q}{\,}\sum_{j=0}^{q-1}\cos\left(j\frac{2n\pi}{q}\right).
\end{equation}
Alternatively, we also have the following expression (see also \cite[p.~41]{Ela05})
$$
g_0(n) ~=~
\begin{cases}
\frac{1}{q}+\frac{2}{q}\,\sum_{j=1}^{(q-1)/2}\cos(j\frac{2n\pi}{q}), & \text{if $q$ is odd},\\
\frac{1}{q}+\frac{1}{q}(-1)^n+\frac{2}{q}\,\sum_{j=1}^{(q/2)-1}\cos(j\frac{2n\pi}{q}), & \text{if $q$ is even},
\end{cases}
$$
or equivalently,
$$%\begin{equation}\label{eq:g02}
g_0(n) ~=~ \frac{1}{q}+\frac{(-1)^n+(-1)^{n+q}}{2q}+\frac{2}{q}\,\sum_{j=1}^{\lfloor (q-1)/2\rfloor}\cos\left(j\frac{2n\pi}{q}\right).
$$%\end{equation}

%---------------------------------------------------------------------------------------
\section{Some applications}

In this section we consider some examples to illustrate and demonstrate the use of Propositions~\ref{prop:main} and \ref{prop:gf}. A few of these examples make use of the \emph{harmonic number} with a complex argument, which is defined by the series
$$
H_z ~=~ \sum_{n\geq 1}\left(\frac{1}{n}-\frac{1}{n+z}\right),\qquad z\in\C\setminus\{-1,-2,\ldots\},
$$
(see, e.g., \cite[p.~311, Ex.~6.22]{GraKnuPat94} and \cite[p.~95, Ex.~19]{Knu97}).

\begin{remark} Formula~\eqref{eq:thmmain} is clearly helpful to obtain an explicit expression for the sum $S_p(n)$ whenever closed-form representations of the associated sums $T_p(n)$ for $p=0,\ldots,q-1$ are available. Otherwise, the formula might be of little interest. For instance, the formula will not be very useful to obtain an explicit expression for the \emph{number of derangements} (see, e.g., \cite[p.~195]{GraKnuPat94})
$$
d(n) ~=~ n!\,\sum_{k=0}^n(-1)^k\,\frac{1}{k!}{\,}.
$$
Indeed, the associated sums $\sum_{k=0}^n\frac{1}{(2k)!}$ and $\sum_{k=0}^n\frac{1}{(2k+1)!}$ have no known closed-form representations.
\end{remark}

%-----------------------------------------------------------------
\subsection{Sums weighted by a $4$-periodic sequence}

Suppose we wish to provide a closed-form representation of the sum
$$
S(n) ~=~ \sum_{k=0}^{n-1}\sin\left(k\frac{\pi}{2}\right)f(k){\,},\qquad n\in\N,
$$
where the function $g(k)=\sin(k\frac{\pi}{2})$ is $4$-periodic. By \eqref{eq:SSp} we then have
$$
S(n) ~=~ \sum_{p=0}^{3}\sin\left(p\frac{\pi}{2}\right)S_p(n) ~=~ S_1(n)-S_3(n){\,},\qquad n\in\N,
$$
where
$$
S_p(n) ~=~ \sum_{k=0}^{n-1}g_0(k-p)f(k){\,},\qquad p\in\{1,3\},
$$
and
$$
g_0(n) ~=~ \frac{1}{4}+\frac{1}{4}(-1)^n+\frac{1}{2}\,\cos\left(n\frac{\pi}{2}\right),\qquad n\in\Z.
$$
Now, if an explicit expression for the sum $T_p(n)=\sum_{k=0}^{n-1}f(4k+p)$ for any $p\in\{1,3\}$ is available, then a closed-form expression for $S_p(n)$ can be immediately obtained by \eqref{eq:thmmain}.

Also, if $F(z)$ denotes the generating function of the sequence $(f(n))_{n\geq 0}$, then by Proposition~\ref{prop:gf} the generating function of the sequence $(S(n))_{n\geq 0}$ is simply given by
$$
\frac{iz}{2(1-z)}\big(F(-iz)-F(iz)\big).
$$

To illustrate, let us consider a few examples.

\begin{example}
Suppose that $f(n)=\log(n+1)$ for all $n\in\N$. It is not difficult to see that
$$
T_p(n) ~=~ \log\left(4^n{\,}\frac{\Gamma\left(n+\frac{p+1}{4}\right)}{\Gamma\left(\frac{p+1}{4}\right)}\right){\,},\qquad p\in\{1,3\}.
$$
Defining $T_p^+$ on $D_p=\{\frac{-p}{4},\frac{1-p}{4},\frac{2-p}{4},\ldots\}$ by
$$
T_p(x) ~=~ \log\left(4^x{\,}\frac{\Gamma\left(x+\frac{p+1}{4}\right)}{\Gamma\left(\frac{p+1}{4}\right)}\right){\,},\qquad p\in\{1,3\},
$$
and then using \eqref{eq:thmmain} we obtain
$$
S_p(n) ~=~ \sum_{k=0}^{3} g_0(n+k-p)\,\log\left(4^{\frac{n+k-p}{4}}{\,}\frac{\Gamma\left(\frac{n+k+1}{4}\right)}{\Gamma\left(\frac{p+1}{4}\right)}\right){\,},\qquad p\in\{1,3\}.
$$
Since $S(n)=S_1(n)-S_3(n)$, after some algebra we finally obtain
$$
S(n) ~=~ \log\left(\frac{2}{\sqrt{\pi}}\right)
+\cos\left(n\frac{\pi}{2}\right)\,\log\left(\frac{\Gamma\left(\frac{n+2}{4}\right)}{2\,\Gamma\left(\frac{n+4}{4}\right)}\right)
+\sin\left(n\frac{\pi}{2}\right)\,\log\left(\frac{\Gamma\left(\frac{n+1}{4}\right)}{2\,\Gamma\left(\frac{n+3}{4}\right)}\right).
$$
\end{example}

\begin{example}\label{ex:over}
Suppose that $f(n)=\frac{1}{n+1}$ for all $n\in\N$. Here, one can show that
$$
T_p(n) ~=~ \frac{1}{4}\left(H_{n+\frac{p+1}{4}-1}-H_{\frac{p+1}{4}-1}\right){\,},\qquad p\in\{1,3\}.
$$
Defining $T_p^+$ on $D_p=\{\frac{-p}{4},\frac{1-p}{4},\frac{2-p}{4},\ldots\}$ by
$$
T_p(x) ~=~ \frac{1}{4}\left(H_{x+\frac{p+1}{4}-1}-H_{\frac{p+1}{4}-1}\right){\,},\qquad p\in\{1,3\},
$$
and then using \eqref{eq:thmmain} we obtain
$$
S_p(n) ~=~ \sum_{k=0}^{3} g_0(n+k-p)\,\frac{1}{4}\left(H_{\frac{n+k+1}{4}-1}-H_{\frac{p+1}{4}-1}\right){\,},\qquad p\in\{1,3\}.
$$
After simplifying the resulting expressions, we finally obtain
$$
S(n) ~=~ \frac{1}{4}\,\log 4
+\frac{1}{4}\,\cos\left(n\frac{\pi}{2}\right)\left(H_{\frac{n-2}{4}}-H_{\frac{n}{4}}\right)
+\frac{1}{4}\,\sin\left(n\frac{\pi}{2}\right)\left(H_{\frac{n-3}{4}}-H_{\frac{n-1}{4}}\right).
$$
Since $F(z)=-\frac{1}{z}\log(1-z)$, the generating function of the sequence $(S(n))_{n\geq 0}$ is given by
$$
\frac{1}{2(1-z)}\big(\log(1-iz)+\log(1+iz)\big).
$$
\end{example}

\begin{example}
Suppose that $f(n)=H_n$ for all $n\in\N$. That is, we are to evaluate the sum
$$
S(n) ~=~ \sum_{k=0}^{n-1}\sin\left(k\frac{\pi}{2}\right)H_k{\,},\qquad n\in\N.
$$
Using summation by parts we obtain
\begin{eqnarray*}
S(n) &=& -\frac{1}{\sqrt{2}}\,\sum_{k=0}^{n-1}\left(\Delta_k\sin\left(k\frac{\pi}{2}+\frac{\pi}{4}\right)\right)H_k\\
&=& -\frac{1}{\sqrt{2}}\,\sin\left(n\frac{\pi}{2}+\frac{\pi}{4}\right)H_n
+\frac{1}{\sqrt{2}}\,\sum_{k=0}^{n-1}\sin\left(k\frac{\pi}{2}+\frac{3\pi}{4}\right)\frac{1}{k+1}{\,},
\end{eqnarray*}
where the latter sum can be evaluated as in Example~\ref{ex:over}. After simplification we obtain
\begin{eqnarray*}
S(n) &=& \frac{\pi-2\ln 2}{8}+\frac{1}{8}\,\cos\left(n\frac{\pi}{2}\right)
\left(H_{\frac{n}{4}}-H_{\frac{n-1}{4}}-H_{\frac{n-2}{4}}+H_{\frac{n-3}{4}}-4H_n\right)\\
&& \null +\frac{1}{8}\,\sin\left(n\frac{\pi}{2}\right)
\left(H_{\frac{n}{4}}+H_{\frac{n-1}{4}}-H_{\frac{n-2}{4}}-H_{\frac{n-3}{4}}-4H_n\right).
\end{eqnarray*}

Using the classical multiplication formula (see, e.g., \cite[{\S}6.4.8]{AbrSte72})
$$
4H_x ~=~ H_{\frac{x}{4}}+H_{\frac{x-1}{4}}+H_{\frac{x-2}{4}}+H_{\frac{x-3}{4}}+4\ln 4
$$
we finally obtain
\begin{eqnarray*}
S(n) &=& \frac{\pi-2\ln 2}{8}-\frac{1}{4}\,\cos\left(n\frac{\pi}{2}\right)
\left(H_{\frac{n-1}{4}}+H_{\frac{n-2}{4}}+4\ln 2\right)\\
&& \null -\frac{1}{4}\,\sin\left(n\frac{\pi}{2}\right)
\left(H_{\frac{n-2}{4}}+H_{\frac{n-3}{4}}+4\ln 2\right).
\end{eqnarray*}
\end{example}

\begin{example}
Let us consider the following series
$$
S ~=~ \sum_{k=1}^{\infty}\sin\left(k\frac{\pi}{2}\right)\frac{1}{k^2}{\,}.
$$
In this case the function $f\colon\N\to\R$ is defined by $f(0)=0$ and $f(k)=1/k^2$ for all $k\geq 1$. By using the identity $S_p(4n)=T_p(n)$ (see Proposition~\ref{prop:main}) we immediately obtain
$$
S ~=~ S_1(\infty)-S_3(\infty) ~=~ T_1(\infty)-T_3(\infty) ~=~ \sum_{k=0}^{\infty}\left(\frac{1}{(4k+1)^2}-\frac{1}{(4k+3)^2}\right).
$$
Actually, this expression is nothing other than the Catalan constant
$$
G ~=~ \sum_{k=0}^{\infty}(-1)^k\frac{1}{(2k+1)^2} ~=~ 0.915965594\ldots
$$
\end{example}

%-----------------------------------------------------------------
\subsection{Alternating sums}

Consider the alternating sum
$$
S(n) ~=~ \sum_{k=0}^{n-1}(-1)^kf(k),\qquad n\in\N.
$$
Here we clearly have $q=2$. Using \eqref{eq:SSp} and \eqref{eq:thmmain} we immediately obtain
\begin{eqnarray*}
S(n) ~=~ S_0(n)-S_1(n)
&=& g_0(n){\,}T_0^+\Big(\frac{n}{2}\Big)+g_0(n+1){\,}T_0^+\Big(\frac{n+1}{2}\Big)\\
&& \null -g_0(n-1){\,}T_1^+\Big(\frac{n-1}{2}\Big)-g_0(n){\,}T_1^+\Big(\frac{n}{2}\Big),
\end{eqnarray*}
which requires the explicit computation of the sums $T_0$ and $T_1$. Alternatively, since $(-1)^k=2{\,}g_0(k)-1$, setting $S_f(n)=\sum_{k=0}^{n-1}f(k)$ we also have
$$
S(n) ~=~ 2{\,}S_0(n)-S_f(n) ~=~  2{\,}g_0(n){\,}T_0^+\Big(\frac{n}{2}\Big)+2{\,}g_0(n+1){\,}T_0^+\Big(\frac{n+1}{2}\Big)-S_f(n),
$$
that is,
\begin{equation}\label{eq:7sdf}
S(n) ~=~ \left(T_0^+\Big(\frac{n}{2}\Big)+T_0^+\Big(\frac{n+1}{2}\Big)-S_f(n)\right)
+(-1)^n\left(T_0^+\Big(\frac{n}{2}\Big)-T_0^+\Big(\frac{n+1}{2}\Big)\right),
\end{equation}
which requires the explicit computation of the sums $T_0$ and $S_f$. It is then easy to compute the sum $T_1$ since by Proposition~\ref{prop:main} we have
$$
T_1(n) ~=~ S_1(2n) ~=~ \frac{1}{2}(S_f(2n)-S(2n)).
$$

The following proposition shows that the expression $T_0^+\big(\frac{n}{2}\big)+T_0^+\big(\frac{n+1}{2}\big)-S_f(n)$ in \eqref{eq:7sdf} can be made independent of $n$ by choosing an appropriate extension $T_0^+$. In this case, that expression is simply given by the constant $T_0^+(\frac{1}{2})$ and hence no longer requires the computation of $S_f(n)$.

\begin{proposition}\label{prop:7sdf}
The following conditions are equivalent.
\begin{enumerate}
\item[(i)] $T_0^+(\frac{n}{2})+T_0^+(\frac{n+1}{2})-S_f(n)$ is constant on $\N$.
\item[(ii)] We have $T_0^+(\frac{n}{2}+1)-T_0^+(\frac{n}{2}) ~=~ f(n)$ for all odd $n\in\N$.
\item[(iii)] $T_0^+(n+\frac{1}{2})-T_1(n)$ is constant on $\N$.
\end{enumerate}
\end{proposition}

\begin{proof}
Assertion (i) holds if and only if
$$
\Delta_n\left(T_0^+\Big(\frac{n}{2}\Big)+T_0^+\Big(\frac{n+1}{2}\Big)\right) ~=~ f(n),\qquad n\in\N,
$$
or equivalently,
$$
T_0^+\Big(\frac{n}{2}+1\Big)-T_0^+\Big(\frac{n}{2}\Big) ~=~ f(n),\qquad n\in\N,
$$
where the latter identity holds whenever $n$ is even. This proves the equivalence between assertions (i) and (ii).

Replacing $n$ by $2n+1$ in the latter identity, we see that assertion (ii) is equivalent to
$$
T_0^+\Big(n+\frac{3}{2}\Big)-T_0^+\Big(n+\frac{1}{2}\Big) ~=~ f(2n+1),\qquad n\in\N,
$$
that is,
$$
\Delta_n T_0^+\Big(n+\frac{1}{2}\Big) ~=~ \Delta_n T_1(n),\qquad n\in\N.
$$
This proves the equivalence between assertions (ii) and (iii).
\end{proof}

Interestingly, we also have
$$
\sum_{k=0}^{2n}(-1)^kf(k) ~=~ \sum_{k=0}^nf(2k)-\sum_{k=0}^{n-1}f(2k+1) ~=~ T_0(n+1)-T_1(n).
$$

Also, if $F(z)$ denotes the generating function of the sequence $(f(n))_{n\geq 0}$, then by Proposition~\ref{prop:gf} the generating function of the sequence $(S(n))_{n\geq 0}$ is simply given by $\frac{z}{1-z}{\,}F(-z)$.

\begin{example}\label{ex:alt1k}
Let us provide an explicit expression for the sum
$$
\sum_{k=1}^{n-1}(-1)^k{\,}\frac{1}{k}{\,},\qquad n\in\N\setminus\{0\}.
$$
Let $f\colon\N\to\R$ be defined by $f(0)=0$ and $f(k)=\frac{1}{k}$ for all $k\geq 1$. Then $S_f\colon\N\to\R$ is defined by $S_f(0)=0$ and $S_f(n)=\sum_{k=1}^{n-1}\frac{1}{k} =H_{n-1}$ for all $n\geq 1$. Also, $T_0\colon\N\to\R$ is defined by $T_0(0)=0$ and $T_0(n) = \sum_{k=1}^{n-1}\frac{1}{2k} = \frac{1}{2}{\,}H_{n-1}$ for all $n\geq 1$. Finally, define the function $T_0^+$ on $D_0=\frac{1}{2}\N=\{\frac{0}{2},\frac{1}{2},\frac{2}{2},\ldots\}$ by $T_0^+(0)=0$ and $T_0^+(x) = \frac{1}{2}{\,}H_{x-1}$ if $x>0$. It is then easy to see that
$$
T_0^+\left(\frac{n}{2}+1\right)-T_0^+\left(\frac{n}{2}\right) ~=~ f(n),\qquad n\in\N.
$$
Using \eqref{eq:7sdf} and Proposition~\ref{prop:7sdf}, we finally obtain
$$
\sum_{k=1}^{n-1}(-1)^k{\,}\frac{1}{k} ~=~ -\ln 2+\frac{1}{2}(-1)^n\left(H_{\frac{n-2}{2}}-H_{\frac{n-1}{2}}\right),\qquad n\in\N\setminus\{0\}.
$$
In particular, using the classical duplication formula $2H_x = H_{\frac{x}{2}}+H_{\frac{x-1}{2}}+2\ln 2$ (see, e.g., \cite[{\S}6.3.8]{AbrSte72}), we obtain
$$
\sum_{k=1}^{2n}(-1)^{k+1}{\,}\frac{1}{k} ~=~ H_{2n}-H_n ~=~ \sum_{k=1}^n\frac{1}{n+k}{\,},\qquad n\in\N.
$$
Also, since $F(z)=-\log(1-z)$, the generating function of the sequence $(S(n))_{n\geq 0}$ is given by $-\frac{z}{1-z}\,\log(1+z)$.
\end{example}

%-----------------------------------------------------------------
\subsection{Sums weighted by a $3$-periodic sequence}

Suppose we wish to evaluate the sum
$$
S(n) ~=~ \sum_{k=0}^{n-1}\cos\left(\frac{2k\pi}{3}\right)f(k),\qquad n\in\N,
$$
where the function $g(k)=\cos(\frac{2k\pi}{3})$ is 3-periodic. By \eqref{eq:SSp}, we have
$$
S(n) ~=~ S_0(n)-\frac{1}{2}S_1(n)-\frac{1}{2}S_2(n),
$$
which, using \eqref{eq:thmmain}, requires the explicit computation of the sums $T_0$, $T_1$, and $T_2$. Alternatively, since $\cos(\frac{2k\pi}{3})=\frac{3}{2}g_0(k)-\frac{1}{2}$, setting $S_f(n)=\sum_{k=0}^{n-1}f(k)$ we also have
$$
S(n) ~=~ \frac{3}{2}S_0(n)-\frac{1}{2}S_f(n),
$$
that is, using \eqref{eq:thmmain},
\begin{eqnarray*}
S(n) &=& \frac{1}{2}\left(T_0^+\Big(\frac{n}{3}\Big)+T_0^+\Big(\frac{n+1}{3}\Big)+T_0^+\Big(\frac{n+2}{3}\Big)-S_f(n)\right)\\
&& \null + \cos\Big(\frac{2n\pi}{3}\Big)\left(T_0^+\Big(\frac{n}{3}\Big)-\frac{1}{2}{\,}T_0^+\Big(\frac{n+1}{3}\Big)
-\frac{1}{2}{\,}T_0^+\Big(\frac{n+2}{3}\Big)\right)\\
&& \null + \sin\Big(\frac{2n\pi}{3}\Big)\left(-\frac{\sqrt{3}}{2}{\,}T_0^+\Big(\frac{n+1}{3}\Big)
+\frac{\sqrt{3}}{2}{\,}T_0^+\Big(\frac{n+2}{3}\Big)\right),
\end{eqnarray*}
which requires the explicit computation of the sums $T_0$ and $S_f$ only.

We then have the following proposition, which can be established in the same way as Proposition~\ref{prop:7sdf}. We thus omit the proof.

\begin{proposition}\label{prop:7sdf3}
The following conditions are equivalent.
\begin{enumerate}
\item[(i)] $T_0^+(\frac{n}{3})+T_0^+(\frac{n+1}{3})+T_0^+(\frac{n+2}{3})-S_f(n)$ is constant on $\N$.
\item[(ii)] For any $p\in\{1,2\}$, we have $(\Delta{\,}T_0^+)(n+\frac{p}{3})=f(3n+p)$.
\item[(iii)] For any $p\in\{1,2\}$, $T_0^+(n+\frac{p}{3})-T_p(n)$ is constant on $\N$.
\end{enumerate}
\end{proposition}

Also, if $F(z)$ denotes the generating function of the sequence $(f(n))_{n\geq 0}$, then by Proposition~\ref{prop:gf} the generating function of the sequence $(S(n))_{n\geq 0}$ is simply given by $\frac{z}{2(1-z)}{\,}(F(\omega z)+F(\omega^{-1}z))$.

\begin{example}
Let us compute the sum
$$
S(n) ~=~ \sum_{k=1}^{n-1}\cos\left(\frac{2k\pi}{3}\right)\log(k),\qquad n\in\N\setminus\{0\}.
$$
Let $f\colon\N\to\R$ be defined by $f(0)=0$ and $f(k)=\log(k)$ for all $k\geq 1$. As observed in Example~\ref{ex:1}, $T_0\colon\N\to\R$ is defined by $T_0(0)=0$ and $T_0(n)=\log(3^{n-1}\Gamma(n))$ for all $n\geq 1$. Also, define the function $T_0^+$ on $D_0=\frac{1}{3}\N$ by $T_0^+(0)=0$ and $T_0^+(x)=\log(3^{x-1}\Gamma(x))$ if $x>0$. We then see that condition (ii) of Proposition~\ref{prop:7sdf3} holds.

Finally, after some algebra we obtain
\begin{eqnarray*}
S(n) &=& \log\left(\frac{\sqrt{2\pi}}{3^{3/4}}\right)
+\frac{1}{2}\,\cos\Big(\frac{2n\pi}{3}\Big)\,\log\left(\frac{3^{n-3/2}\,\Gamma(\frac{n}{3})^3}{2\pi\,\Gamma(n)}\right)\\
&& \null +\frac{\sqrt{3}}{2}\,\sin\Big(\frac{2n\pi}{3}\Big)\,\log\left(\frac{3^{1/3}\,\Gamma(\frac{n+2}{3})}{\Gamma(\frac{n+1}{3})}\right),
\end{eqnarray*}
which, put in another form, is the function obtained in Example~\ref{ex:1}.
\end{example}

%-----------------------------------------------------------------
\subsection{Computing $T_p(n)$ from $S_p(n)$}

We now present two examples where the computation of the sum $T_p(n)$ can be made easier by first computing the sum $S_p(n)$. According to Proposition~\ref{prop:main}, the corresponding conversion formula is simply given by $T_p(n)=S_p(qn)$ for all $n\in\N$.

\begin{example}\label{ex:bin}
Let $m\in\N\setminus\{0\}$ and $p\in\{0,\ldots,q-1\}$. Suppose we wish to provide a closed-form representation of the sum
$$
\sum_{k\geq 0}{m\choose q k+p}h(q k+p){\,}.
$$
for some function $h\colon\N\to\C$. Considering the function $f(k)={m\choose k}h(k)$, the sum above is nothing other than $T_p(n)$ for any $n\geq\lfloor\frac{m-p}{q}\rfloor +1$. Actually, we can even write
$$
\sum_{k\geq 0}{m\choose q k+p}h(q k+p) ~=~ T_p(\infty) ~=~ S_p(\infty) ~=~ \sum_{k\geq 0}{m\choose k}g_0(k-p)h(k){\,}.
$$
Using \eqref{eq:g01}, the latter expression then becomes
$$
\frac{1}{q}{\,}\sum_{j=0}^{q-1}{\,}\sum_{k\geq 0}{m\choose k}\omega^{j(k-p)}h(k) ~=~ \frac{1}{q}{\,}\sum_{j=0}^{q-1}\omega^{-jp}{\,}\sum_{k\geq 0}{m\choose k}\omega^{jk}{\,}h(k),
$$
where the inner sum can sometimes be evaluated explicitly.

For instance, if $h(k)=1$ (see, e.g., \cite[p.~71, Ex.~38]{Knu97}), then using the classical identity $e^{i\theta}+1=2{\,}e^{i\frac{\theta}{2}}\cos\frac{\theta}{2}$ the latter expression reduces to
$$
\frac{1}{q}{\,}\sum_{j=0}^{q-1}\omega^{-jp}{\,}(\omega^j+1)^m ~=~ \frac{2^m}{q}{\,}\sum_{j=0}^{q-1}e^{i\frac{j(m-2p)\pi}{q}}\cos^m\frac{j\pi}{q}{\,}.
$$
Since this quantity is known to be real, we may take the real part and finally obtain
$$
\sum_{k\geq 0}{m\choose q k+p} ~=~ \frac{2^m}{q}{\,}\sum_{j=0}^{q-1}\cos\frac{j(m-2p)\pi}{q}\,\cos^m\frac{j\pi}{q}{\,}.
$$

To give a second example, if $h(k)=\frac{1}{k+1}=\int_0^1x^k{\,}dx$, then we proceed similarly and obtain
%$$
%\frac{1}{q}{\,}\sum_{j=0}^{q-1}\omega^{-jp}\int_0^1(\omega^jx+1)^m{\,}dx
%~=~ \frac{1}{q(m+1)}{\,}\sum_{j=0}^{q-1}\big(\omega^{-j(p+1)}\big((\omega^j+1)^{m+1}-1\big)\big).
%$$
\begin{multline*}
\sum_{k\geq 0}{m\choose qk+p}\frac{1}{qk+p+1} ~=\\ \frac{1}{q(m+1)}{\,}\sum_{j=0}^{q-1}\left(2^{m+1}\cos\frac{j(m-2p-1)\pi}{q}\,\cos^{m+1}\frac{j\pi}{q}-\cos\frac{2j(p+1)\pi}{q}\right).
\end{multline*}

The case where $m$ is not an integer is more delicate. Fixing $z\in\C$ and letting $f(k)={z\choose k}$, we have for instance
$$
\sum_{k=0}^{n-1}{z\choose 2k+1} ~=~ T_1(n) ~=~ S_1(2n) ~=~ S_1(2n+1) ~=~ \sum_{k=0}^{2n}\frac{1-(-1)^k}{2}{\,}{z\choose k}.
$$
Using the identity $\sum_{k=0}^n(-1)^k{z\choose k}=(-1)^n{z-1\choose n}$ (see, e.g., \cite[p.~165]{GraKnuPat94}), we obtain
$$
\sum_{k=0}^{n-1}{z\choose 2k+1} ~=~ \frac{1}{2}\,\sum_{k=0}^{2n}{z\choose k}-\frac{1}{2}{\,}{z-1\choose 2n},
$$
which is not a closed-form expression.
\end{example}

\begin{example}%[Values of the harmonic number for fractional arguments]
Let us illustrate the use of Proposition~\ref{prop:main} by proving the following Gauss formula, which provides an explicit representation of the harmonic number for fractional arguments (see, e.g., \cite[p.~95]{Knu97} and \cite[p.~30]{SriCho12}). For any integer $p\in\{1,\ldots,q-1\}$, we have
$$
H_{\frac{p}{q}} ~=~ \frac{q}{p}-\ln(2q)-\frac{\pi}{2}\cot\frac{p\pi}{q}+2\sum_{j=1}^{\lfloor (q-1)/2\rfloor}\,\cos\left(\frac{2jp\pi}{q}\right)\ln\left(\sin\frac{j\pi}{q}\right).
$$

To establish this formula, define $f\colon\N\to\R$ by $f(0)=0$ and $f(k)=\frac{1}{k}$ for $k\geq 1$. We then have
\begin{eqnarray*}
\frac{1}{q}{\,}H_{\frac{p}{q}} &=& \sum_{n\geq 1}\left(\frac{1}{q n}-\frac{1}{q n+p}\right)
~=~ \frac{1}{p}+\lim_{N\to\infty}\left(\sum_{n=1}^{N-1}\frac{1}{q n}-\sum_{n=0}^{N-1}\frac{1}{q n+p}\right)\\
&=& \frac{1}{p}+\lim_{N\to\infty} (T_0(N)-T_p(N))\\
&=& \frac{1}{p}+\lim_{N\to\infty} (S_0(qN)-S_p(qN+p))\\
&=& \frac{1}{p}+\lim_{N\to\infty} \left(\sum_{n=1}^{qN-1}\frac{g_0(n)-g_0(n-p)}{n}-\sum_{n=qN}^{qN+p-1}\frac{g_0(n-p)}{p}\right)\\
&=& \frac{1}{p}+\sum_{n\geq 1}\frac{1}{n}(g_0(n)-g_0(n-p)){\,},
\end{eqnarray*}
that is, using \eqref{eq:g01},
$$
H_{\frac{p}{q}} ~=~  \frac{q}{p}+\sum_{j=1}^{q-1}(1-\omega^{-jp})\sum_{n\geq 1}\frac{\omega^{jn}}{n}{\,}.
$$
Since $\omega^j\neq 1$ for $j=1,\ldots,q-1$, the inner series converges to $-\log(1-\omega^j)$, where the complex logarithm satisfies $\log 1=0$.

Using the identity $1-e^{i\theta}=-2{\,}i{\,}e^{i\frac{\theta}{2}}\sin\frac{\theta}{2}$, we obtain
$$
\log(1-\omega^j) ~=~ \ln\left(2\sin\frac{j\pi}{q}\right)+i\left(\frac{j\pi}{q}-\frac{\pi}{2}\right).
$$

It follows that
\begin{eqnarray*}
H_{\frac{p}{q}} &=& \frac{q}{p}-\sum_{j=1}^{q-1}\mathrm{Re}\left(\left(1-e^{-i\frac{2jp\pi}{q}}\right)\left(\ln\left(2\sin\frac{j\pi}{q}\right)+i\left(\frac{j\pi}{q}-\frac{\pi}{2}\right)\right)\right)\\
&=& \frac{q}{p}-\sum_{j=1}^{q-1}\left(\left(1-\cos\frac{2jp\pi}{q}\right)\ln\left(2\sin\frac{j\pi}{q}\right)
-\left(\frac{j\pi}{q}-\frac{\pi}{2}\right)\sin\frac{2jp\pi}{q}\right).
\end{eqnarray*}

Now, it is not difficult to show that
$$
\sum_{j=1}^{q-1}\sin\frac{2jp\pi }{q}~=~0,\qquad\sum_{j=1}^{q-1}\cos\frac{2jp\pi }{q}~=~-1,
$$
and
$$
\sum_{j=1}^{q-1}j{\,}\sin\frac{2jp\pi }{q}~=~ -\frac{q}{2}{\,}\cot\frac{p\pi}{q}{\,}.
$$
Thus, we obtain
\begin{eqnarray*}
H_{\frac{p}{q}} &=& \frac{q}{p}-q\ln 2-\frac{\pi}{2}\cot\frac{p\pi}{q} -\ln\bigg(\prod_{j=1}^{q-1}\sin\frac{j\pi}{q}\bigg)+\sum_{j=1}^{q-1}\cos\frac{2jp\pi}{q}\,\ln\bigg(\sin\frac{j\pi}{q}\bigg),
\end{eqnarray*}
where the product of sines, which can be evaluated by means of Euler's reflection formula and then the multiplication theorem for the gamma function, is exactly $q{\,}2^{1-q}$. Finally, the expression for $H_{\frac{p}{q}}$ above reduces to
$$
H_{\frac{p}{q}} ~=~ \frac{q}{p}-\ln (2q)-\frac{\pi}{2}\cot\frac{p\pi}{q} +\sum_{j=1}^{q-1}\cos\frac{2jp\pi }{q}\,\ln\left(\sin\frac{j\pi}{q}\right).
$$
To conclude the proof, we simply observe that both $\cos\frac{2jp\pi }{q}$ and $\sin\frac{j\pi}{q}$ remain invariant when $j$ is replaced with $q-j$.
\end{example}

%---------------------------------------------------------------------------------------------- Acknowledgments
\section*{Acknowledgments}

This research is supported by the internal research project R-AGR-0500 of the University of Luxembourg.

%---------------------------------------------------------------------------------------

\end{document}